\documentclass[12pt,reqno]{amsart}

\usepackage[arrow,matrix,curve]{xy}

\usepackage[dvips]{graphicx} 
\usepackage{mathtools}

\usepackage{amssymb, latexsym, amsmath, amscd, array, hyperref, esint
}

\usepackage{breakurl}   

\makeatletter
\@namedef{subjclassname@2020}{\textup{2020} Mathematics Subject Classification}
\makeatother

\newtheorem{theorem}{Theorem}[section]

\newtheorem{corollary}[theorem]{Corollary}

\theoremstyle{definition}

\numberwithin{equation}{section}

\newcommand\R {{\mathbb R}}

\newcommand\ie{{\it i.e.}}

\DeclareMathOperator{\area}{{\rm area}} 
\DeclareMathOperator{\vol}{{\rm vol}} 
\DeclareMathOperator{\length}{{\rm length}} 
\DeclareMathOperator{\sys}{{\rm sys}}
\DeclareMathOperator{\diam}{{\rm diam}}
\DeclareMathOperator{\Var}{{\rm Var}}

\newcommand{\RP}{{\mathbb R\mathbb P}}

\long\def\forget#1\forgotten{} %

\numberwithin{equation}{section}

\title {A Pu--Bonnesen inequality}

\author[M. Katz]{Mikhail G. Katz} \address{M.~Katz, Department of
  Mathematics, Bar Ilan University, Ramat Gan 5290002 Israel}
\email{katzmik@math.biu.ac.il}

\author[S.\;Sabourau]{St\'ephane Sabourau}
\address{\parbox{\linewidth}{Univ Paris Est Creteil, CNRS, LAMA, F-94010 Creteil, France \\
Univ Gustave Eiffel, LAMA, F-77447 Marne-la-Vall\'ee, France}}
\email{stephane.sabourau@u-pec.fr}

\subjclass[2020]
{Primary 53C45; Secondary 52A38, 53A05, 52A15, 53C20}

\keywords{Convex surfaces, Pu's inequality, Bonnesen's inequality, circumscribed and inscribed radii.}

\thanks{Partially supported by the ANR project Min-Max (ANR-19-CE40-0014).}

\begin{document}

\thispagestyle{empty}


\begin{abstract}
We prove an inequality of Bonnesen type for the real projective plane,
generalizing Pu's systolic inequality for positively-curved metrics.
The remainder term in the inequality, analogous to that in Bonnesen's
inequality, is a function of~$R-r$ (suitably normalized), where~$R$
and~$r$ are respectively the circumradius and the inradius of the
Weyl--Lewy Euclidean embedding of the orientable double cover.  We
exploit John ellipsoids of a convex body and Pogorelov's ridigity
theorem.
\end{abstract}

\maketitle



\section{Pu, Bonnesen, and Weyl--Lewy}
\label{s1}

Pu's inequality asserts that~$\area(g) -\frac2\pi \sys^2(g) \ge0$ for
every Riemannian metric~$g$ on~$\RP^2$ where equality is satisfied if
and only if~$g$ has constant Gaussian curvature (\cite{Pu52}, 1952).
The inequality has recently been strengthened \cite{Ka20} to
\begin{equation}
\label{e11}
\area(g)-\frac{2}{\pi}\sys^2(g)\geq2\pi \Var_\mu(f)
\end{equation}
where the variance $\Var_\mu$ is with respect to the probability
measure~$\mu$ induced by the constant curvature unit area metric~$g_0$
in the conformal class of~$g$, so that~$g=f^2g_0$.  A similar
strengthened version exists for Loewner's torus inequality; see
\cite{Ho09}.  However, the remainder term~$\Var_\mu(f)$ in these
inequalities does not exhibit any explicit relation to the geometry of
the metric.  In this paper we seek a strengthening of Pu's inequality
where the relation to the metric is more explicit.

Bonnesen's inequality and its analogs involve a strengthening of the
isoperimetric inequality of the following type:
\begin{equation}
\label{e12}
L^2 - 4\pi A \geq f(R,r),
\end{equation}
where~$L$ is the length of a Jordan curve in~$\R^2$,~$A$ is the area
of the region bounded by the curve,~$R$ is the circumradius and~$r$ is
the inradius \cite[p.\;3]{BZ}.  In Bonnesen's inequality, one
has~$f(R,r)=\pi^2(R-r)^2$.  Additional inequalities exist with $f$
dependent on parameters $A$ or $L$,
namely~$f(R,r)=A^2\big(\frac1r-\frac1R\big)^2$
and~$f(R,r)=L^2\big(\frac{R-r}{R+r} \big)^2$ (ibid.).

Now suppose~$(\RP^2,g)$ has positive curvature.  Then by the
generalisation of the Weyl--Lewy theorem \cite{We16}, \cite{Le38}, the
orientable double cover~$S_g$ admits a unique (up to congruence)
isometric embedding as a convex surface~$S_g\subseteq\R^3$ 
%
%
(Nirenberg~\cite{Ni53}, Lu~\cite{Lu20}).  Here, we seek an analog of
the inequality~\eqref{e11} with a remainder term exhibiting a more
explicit relation to the geometry of~$S_g$.  The relation is expressed
in terms of the difference~$R-r$ between the circumscribed and
inscribed radii of~$S_g$, analogous to~\eqref{e12}.  Let~$R$ and~$r$
be the circumradius and inradius of~$S_g\subseteq\R^3$, respectively.

\begin{theorem} \label{theo:2d}
\label{t11}
There exists a monotone continuous function~$\lambda(t)>0$ for~$t>0$
such that if~$(\RP^2,g)$ has positive Gaussian curvature, then
\begin{equation}
\label{e13}
\frac{\area(g)}{\sys^2(g)} -\frac{2}{\pi} \geq
\lambda\Big(\frac{R-r}{\sys}\Big),
\end{equation}
where~$\lambda(t)$ is asymptotically linear as~$t\to\infty$, and~$R$
and~$r$ are the circumscribed and inscribed radii of the Euclidean
embedding of the orientable double cover~$S_g \subseteq \R^3$.
\end{theorem}

Our argument produces asymptotically linear bounds for the remainer
term for large~$R-r$ (when the systole is normalized) that can easily
be made effective.  It would be interesting to develop effective lower
bounds for the remainder term for small~$R-r$ as in Volkov
\cite{Vo68}.

Generalisations of Pu's inequality are studied in \cite{Ba98},
\cite{Gu10}, \cite{Iv04}, \cite{Ka06a}, \cite{Ka06b}, and elsewhere.

\section
{Rigidity in Pu's inequality for singular metrics}
\label{sec:rigidity}

In this section, we review some aspects of Reshetnyak's work on
Alexandrov surfaces (\ie, singular surfaces with bounded integral
curvature) focusing on nonnegatively curved metrics.

We refer to \cite{res93}, \cite{res01}, \cite{tro}, \cite{deb} and
\cite{fil} for a more detailed exposition.

Let~$(S^2,g_n)$ be a sequence of nonnegatively curved Riemannian
metrics on~$S^2$ such that the curvature measures~$\omega_n=K_{g_n} \,
dg_n$ weakly converge to~$\omega$.  The measure difference $\mu=\omega
- dg_0$ is a Radon measure of zero total mass on~$S^2$.  Suppose
that~$\mu(\{x\}) < 2 \pi$ for every~$x \in S^2$.  Then~$d_{g_n} \to d$
for the uniform topology; see~\cite[Theorem~6.2]{tro} or~\cite{res60}.
Here~$d$ is the intrinsic distance induced by~$g=e^u g_0$, where~$u$
is the potential of~$\mu$ with respect to the canonical metric~$g_0$
on~$S^2$ (\ie,~$\mu = \Delta_{g_0} u$ in the distribution sense).
Note that $u$ is expressed via the Green function $G:S^2 \times S^2
\to \R \cup \{ + \infty \}$ through the formula
\begin{equation} \label{eq:Green}
u(x) = \int_{S^2} G(x,y) \, d\mu(y).
\end{equation}
For general Alexandrov surfaces, the potential~$u$ is the difference
of two subharmonic functions, and one has~$u \in \cap_{p<2}
W^{1,p}(S^2)$; see \cite{tro} and references therein (this won't be
needed though).  Furthermore,~$\area(g_n) \to \area(g)$ and
\[
\area(g) = \int_{S^2} e^u \, dg_0.
\]

This result readily applies to a sequence of smooth convex spheres
$S_n \subseteq \R^3$ converging to a convex sphere~$S$ (bounding a
convex body) for the Hausdorff topology in~$\R^3$.  Equivalently, it
applies to a sequence of nonnegatively curved metrics~$g_n$ on~$S^2$
Gromov-Hausdorff converging (without collapse) to a metric on~$S^2$
with nonnegative curvature in Alexandrov's sense.  Indeed, in this
case, the distance on~$S_n$ uniformly converges to the distance
on~$S$; see~\cite[\S III.1]{Alex} or~\cite[Lemma~10.2.7]{BBI}.
Furthermore, the condition~$\mu_S(\{x\}) < 2\pi$, where~$x \in S$, is
always satisfied on the boundary~$S$ of a convex body in~$\R^3$;
see~\cite[\S V.3]{Alex}.  Thus, the metric on~$S$ can be written
as~$g=e^u g_0$ for some weakly regular function~$u$ on~$S$ as above.

For such singular metrics, the characterization of the equality case
in Pu's inequality still holds.  More generally, we have

\forget
\begin{theorem}
\label{theo:rigidity}
Let~$\RP^2$ be the projective plane with a singular Riemannian
metric~$g=e^u g_0$, where~$u$ is the difference of two subharmonic
functions.  Then
\[
\area(\RP^2,g) \geq \frac{2}{\pi} \sys(\RP^2,g)^2
\]
with equality if and only if~$g$ is a metric of constant Gaussian
curvature.
\end{theorem}

\begin{proof}
The proof given in~\cite{Ka20} applies in our setting; the bounds
established in~\cite{Ka20} for continuous functions hold true for
integrable functions.  The equality case in the Cauchy-Schwarz
inequality implies that the conformal factor of~$g$ is constant.
\end{proof}
\forgotten

\begin{theorem}
Let~$\RP^n$ be the projective $n$-space with a singular Riemannian
metric~$g=\bar{f}^2 g_0$ with nonzero systole, where $g_0$ is the standard round metric and~$\bar{f} \in L^n$ is a nonnegative function.  
Then
\[
\frac{\vol(\RP^n,g)}{\sys(\RP^n,g)^n} \geq \frac{\vol(\RP^n,g_0)}{\sys(\RP^n,g_0)^n}.
\]
Furthermore, when $n=2$ and $\bar{f}=e^u$, where $u$ is the potential of a Radon measure~$\mu$ of zero total mass given by~\eqref{eq:Green}, equality holds if and only if~$g$ is a metric of constant Gaussian curvature.
\end{theorem}

\begin{proof}
Denote by~$f$ the lift of~$\bar{f}$ to the double orientable cover~$S^n$ of~$\RP^n$.
Let us recall a simple version of Santal\'o's formula on the standard sphere~$S^n$; see~\cite[IV.19.4]{San}.
The space~$\Gamma$ of closed oriented geodesics on~$S^n$ is a $(2n-2)$-dimensional manifold admitting a natural symplectic structure whose corresponding natural volume measure is denoted by~$\nu$.
Every integrable function $F:S^n \to \R$ satisfies
\begin{equation} \label{eq:santalo}
\int_{S^n} F \, dg_{S^n} = \frac{1}{\vol(S^{n-1},g_0)} \, \int_{\gamma \in \Gamma} \int_{S^1} F(\gamma(t)) \, dt \, d\nu(\gamma).
\end{equation}
Taking $F \equiv 1$, we observe that $\nu(\Gamma) = \frac{1}{2\pi} \vol(S^{n-1},g_0) \vol(S^n,g_0)$.

Applied to~$F=f^n$, Santal\'o's formula yields
\[
\vol(S^n,g) = \int_{S^n} f^n \, dg_0 = \frac{1}{\vol(S^{n-1},g_0)} \, \int_{\gamma \in \Gamma} \int_{S^1} f(\gamma(t))^n \, dt \, d\nu(\gamma).
\]
By H\"older's inequality 
\[
(2\pi)^{\frac{n-1}{n}} \, \left( \int_{S^1} f(\gamma(t))^n \, dt \right)^{\frac{1}{n}} \geq \int_{S^1} f(\gamma(t)) \, dt = \length_g(\gamma),
\]
we deduce
\begin{align*}
\vol(S^n,g) & \geq \frac{1}{(2\pi)^{n-1} \vol(S^{n-1},g_0)} \left( \int_{\gamma \in \Gamma} \length_g(\gamma) \, d\nu(\gamma) \right)^n \\
 & \geq \frac{2^n \nu(\Gamma)}{(2\pi)^{n-1} \vol(S^{n-1},g_0)} \, \sys(\RP^n,g)^n
\end{align*}
since $\length_g(\gamma) \geq 2 \, \sys(\RP^n,g)$ for every $\gamma \in \Gamma$.
Combined with the value of~$\nu(\Gamma)$, we derive
\[
\vol(\RP^n,g) \geq \frac{\frac{1}{2} \vol(S^n,g_0)}{\pi^n} \, \sys(\RP^n,g)^n
\]
as desired.
Furthermore, equality holds if and only if equality holds in H\"older's inequality, that is, if $f$ is constant almost everywhere.

Suppose that $n=2$ and $f=e^u$, where~$u$ is given by~\eqref{eq:Green}.
Since~$f$, and so~$u$, are constant almost everywhere, we have $\mu=\Delta_{g_0} u=0$.
By~\eqref{eq:Green}, this implies that $u$, and so~$f$, are constant everywhere.
Hence the extremal metric has constant curvature.
\end{proof}

Specifically, we have

\begin{corollary}
\label{coro:rigidity}
Let~$\RP^2$ be the projective plane with a singular Riemannian
metric~$g=e^u g_0$, where $u$ is the potential of a Radon measure~$\mu$ of zero total mass given by~\eqref{eq:Green}.
Then
\[
\area(\RP^2,g) \geq \frac{2}{\pi} \sys(\RP^2,g)^2
\]
with equality if and only if~$g$ is a metric of constant Gaussian
curvature.
\end{corollary}

\section{Pu's inequality with an~$R-r$ remainder}
\label{s2}

In this section, we present two proofs of Theorem~\ref{theo:2d}.  The
first one follows an extrinsic geometry approach, while the second
relies on intrinsic geometry arguments.

\begin{proof}[First proof of Theorem~\ref{theo:2d}]
We will exploit John ellipsoids as well as Pogorelov's rigidity
theorem.

\medskip
\textbf{Step 1.}  Consider a pair of John ellipsoids~$E\subseteq S_g
\subseteq \sqrt{3} \, E$ as in \cite{Jo48}.  Let~$a\le b \le c$ be the
principal axes of~$E \subseteq \R^3$.  Then up to a uniform
multiplicative constant, each nonnegative curved metric~$g$ on~$\RP^2$
has the following properties:
\begin{enumerate}
\item
\label{i1}
$\sys(g)\sim b$;
\item
the inradius
$r$ of~$S_g$ satisfies~$r\sim a$;
\item
\label{i3}
the circumradius~$R$ of~$S_g$ satisfies~$R\sim c$;
\item
\label{i4}
$\area(g)\sim bc$.
\end{enumerate}
Here, we write~$\alpha \sim \beta$ if there exist two positive
constants~$C$ and~$C'$ (which do not depend on~$g$) such that~$C
\alpha \leq \beta \leq C' \alpha$.

This follows from the existence of distance-decreasing nearest-point
projections from~$\sqrt{3} \, E$ to~$S_g$ and from~$S_g$ to~$E$, due to the
convexity of~$S_g$.

Since the relation~\eqref{e13} is scale-invariant, we can introduce a
normalisation~$b=1$.  Then the systole is uniformly bounded above and
below by item \eqref{i1}.  For a sequence of metrics
with~$c\to\infty$, we have
\mbox{$\area(g)\sim c$} and therefore the inequality~\eqref{e13}
follows from items~\eqref{i3} and~\eqref{i4}, including the linear
asymptotic behavior of the remainder term as~$t\to\infty$.

\medskip
\textbf{Step 2.} By the characterization of the equality case in Pu's
inequality (see Corollary~\ref{coro:rigidity}), it remains to show that
if for a sequence of metrics (with~$b=1$) the
difference\,~$\area-\,\frac{2}{\pi}\sys^2$\, tends to~$0$ then~$S_g$
converges to a round metric.  For such a sequence, the major axis~$c$
is uniformly bounded in view of estimates~\eqref{i1} and~\eqref{i4}.
We need to show that the minor axis~$a$ stays away from zero in such a
sequence of metrics.  Suppose~$a\to0$.  Then~$\sys(g)$ tends to the
width~$W$ of the Jordan curve obtained as the boundary of the
orthogonal projection of~$S_g$ to the plane spanned by the principal
axes~$b$ and~$c$.  Similarly,~$\area(g)$ tends to the area~$A$ of the
planar region~$D \subseteq \R^2$ bounded by the curve.  Since~$D$ is a
centrally symmetric convex planar domain, we have~$W^2\leq \frac4\pi
A$.  Therefore, in the limit, we obtain
\[
\area -\, \frac2\pi\sys^2 = A- \frac2\pi W^2 \geq
A\Big(1-\frac{8}{\pi^2}\Big)>0
\]
which contradicts the assumption that~$\area-\,\frac{2}{\pi}\sys^2$
tends to~$0$.  Hence, the minor axis~$a$ stays uniformly away
from~$0$.

Thus, for a family of metrics with the
difference~$\area-\,\frac{2}{\pi}\sys^2$ tending to~$0$, the
corresponding John ellipsoids have uniformly bounded eccentricity.

\medskip
\textbf{Step 3.}  Convex sets of uniformly bounded eccentricity form a
compact family by the Blaschke selection theorem.  Let~$\lambda(t)$ be
the minimum of~$\frac{\area}{\sys^2}-\frac2\pi$ among nonnegatively
curved metrics with~$\frac{R-r}{\sys}\geq t$.  Every metric which is
not internally isometric to the metric of constant curvature,
satisfies~$R-r>0$.

\medskip
\textbf{Step 4.} Now suppose the metric of the convex surface is
internally isometric to the metric of constant curvature.  It follows
that the surface is congruent to the standard one by Pogorelov's
rigidity theorem~\cite[p.\;167]{Po73} (see Prosanov \cite{Pr20} for a
discussion of the status of the various proofs of this ridigity
result), and therefore $R-r=0$.  Thus~$\lambda(t)>0$ for~$t>0$.
\end{proof}

Our second proof of Theorem~\ref{theo:2d} relies on a more intrinsic
argument.

\begin{proof}[Second proof of Theorem~\ref{theo:2d}]
Since the relation~\eqref{e13} is scale-invariant, we can assume
that~$\sys(g)=1$.

Consider the lift~$\gamma$ on~$S_g \subseteq \R^3$ of a systolic loop of~$(\RP^2,g)$.
Note that~$\gamma$ is a simple closed geodesic of~$S_g$ of length~$L=2 \sys(g)$.
Furthermore, $2\pi r \leq L$.
Thus, $r \leq \frac{1}{\pi}$.
Let~$p$ be a point of~$S_g$ at maximal intrinsic distance from~$\gamma$.
%
%
There exist at least two arcs of length~$D=d_{S_g}(p,\gamma)$ starting at~$p$ and ending perpendicularly at~$\gamma$.
These two arcs along with~$\gamma$ decompose the hemisphere of~$S_g$ bounded by~$\gamma$ into two isosceles triangles~$\Delta_i$.
Now, consider the comparison triangles~$T_i$ with the same side lengths in the Euclidean plane.
By Toponogov's theorem, the area of~$\Delta_i$ is greater or equal to the area of~$T_i$.
Also, one of the edges of one of the triangles~$\Delta_i$ (and so of~$T_i$) is of length at least~$\frac{L}{2}$, and its other two edges are of length~$D$.
If~$D$ goes to infinity (while the systole of~$g$ is fixed), the area of~$T_i$ goes to infinity and so does the area of~$\Delta_i$ (and~$\RP^2$).
More precisely, the area of~$g$ grows roughly as~$D$.
Note that $2R \leq \diam(g) \leq 2D + \frac{L}{2}$.
This implies that the remainder term~$\lambda$ is asymptotically linear in~\eqref{e13}.

Let~$\lambda(t)$ be the infimum of the
difference~$\area-\,\frac{2}{\pi}\sys^2$ among nonnegatively curved
metric with~$\frac{R-r}{\sys}\geq t$.  From the previous discussion,
we can assume that the diameter of~$g$ is bounded when
considering~$\lambda(t)$.  Since the space~$\mathfrak{M}(n,D,v)$
of~$n$-dimensional compact Alexandrov spaces of diameter at most~$D$
and of volume at least~$v >0$ is compact for the Gromov-Hausdorff
topology (see Theorem~10.7.2 and Corollary~10.10.11 of~\cite{BBI}), it
follows from Perelman's Stability Theorem
(see~\cite[Theorem~10.10.5]{BBI}) that this infimum is attained by an
extremal metric with nonnegative curvature in Alexandrov's sense.
Such a (singular) metric can be written as~$g=e^u g_0$, where~$g_0$ is
the canonical metric on~$\RP^2$ and~$u$ is the difference of two
subharmonic functions; see Section~\ref{sec:rigidity}.  Suppose
that~$\lambda(t)=0$.  By the characterization of the equality case in
Pu's inequality (see Corollary~\ref{coro:rigidity}), the extremal
metric~$g$ is internally isometric to a metric of constant Gaussian
curvature.  By Pogorelov's rigidity theorem, such a convex surface is
congruent to the standard round sphere, so that~$R-r=0$.
Therefore,~$\lambda(t)>0$ for~$t>0$.
\end{proof}

\bibliographystyle{amsalpha}
 
\end{document}